\documentclass[11pt]{amsart}

\usepackage{amsmath,amssymb,mathtools,microtype}
\usepackage{booktabs}
\usepackage[colorlinks=true,citecolor=blue,linkcolor=blue,urlcolor=blue]{hyperref}

\newtheorem{theorem}{Theorem}[section]
\newtheorem{proposition}[theorem]{Proposition}
\newtheorem{lemma}[theorem]{Lemma}
\newtheorem{corollary}[theorem]{Corollary}
\newtheorem{conjecture}[theorem]{Conjecture}
\theoremstyle{definition}

\theoremstyle{remark}
\newtheorem{remark}[theorem]{Remark}

\newcommand{\C}{\mathbb{C}}
\newcommand{\D}{\mathbb{D}}
\newcommand{\T}{\mathbb{T}}

\newcommand{\R}{\mathbb{R}}
\newcommand{\PP}{\mathcal{P}}
\newcommand{\HH}{\mathcal{H}}
\newcommand{\Bzero}{\mathcal{B}_0}  
\newcommand{\eqn}{\mathrel{\overset{n}{\sim}}}
\DeclarePairedDelimiter{\abs}{\lvert}{\rvert}
\DeclarePairedDelimiter{\norm}{\lVert}{\rVert}
\DeclareMathOperator{\Real}{Re}

\title[Krzy\.z's conjecture for the sixth coefficient]
{A computer-assisted proof of Krzy\.z's conjecture for the sixth coefficient}

\author{Frank Wikström}
\address{Centre for Mathematical Sciences, Lund University,
Box 118, SE-221 00 Lund, Sweden}
\email{frank.wikstrom@math.lth.se}


\begin{document}

\begin{abstract}
We develop a computer-assisted approach to Krzy\.z's conjecture based on
validated elimination of the masses in an extremal Herglotz representation.
The entropy theorem of Lei and Zhang, together with the Agler--McCarthy
reduction, leaves only extremals with fewer than $n$ atoms to exclude.  In
the critical stratum with $n-1$ atoms, the first-order conditions are linear
in the masses.  Validated local inversion reduces the regular search to
$n$ outer parameters, with a rank-free fallback at singular configurations.
This gives a proof of the sixth-coefficient conjecture, including its
equality statement; the critical-stratum computation takes 8.2 minutes on
an M4 Pro.  We also obtain an independent proof for the seventh coefficient,
recently established by Xie.  The computational exclusions use explicit
rounding-error bounds, and the equality cases follow analytically.
\end{abstract}

\maketitle

\section{Introduction}\label{sec:intro}

We present a computer-assisted proof of the sixth-coefficient case of
Krzy\.z's conjecture and an independent proof of the seventh-coefficient
case, recently established by Xie~\cite{Xie2026}.  The main computational
ingredient is validated mass elimination: the first-order conditions for
extremality determine the masses locally from a smaller set of parameters.
Combined with the recent solution of the Agler--McCarthy entropy conjecture,
this reduces the problem to finite searches away from the equality cases.

Let $\Bzero$ be the set of holomorphic functions $f\colon\D\to\D\setminus\{0\}$
and write $f(z)=\sum_{k\ge0}a_k(f)z^k$.  For $n\ge1$ put
\[
M_n=\sup_{f\in\Bzero}\abs{a_n(f)} .
\]
A normal-families argument shows that the supremum is attained and a function
attaining it is called \emph{$M_n$-extremal}.  The function
\[
f_n(z)=\exp\Bigl(-\frac{1-z^n}{1+z^n}\Bigr)=\frac1e\bigl(1+2z^n+\dotsb\bigr)
\]
shows that $M_n\ge 2/e$.

\begin{conjecture}[Krzy\.z \cite{Krzyz1968}]\label{conj:krzyz} $M_n=2/e$ for
every $n\ge1$, and the extremal functions are exactly $\zeta f_n(\varepsilon z)$
with $\abs{\zeta}=\abs{\varepsilon}=1$.
\end{conjecture}

The case $n=1$ is elementary.  The conjecture was proved for $n=2$ by Krzy\.z,
Reade and MacGregor, for $n=3$ by Hummel, Scheinberg and Zalcman \cite{HSZ1977},
for $n=4$ by Tan \cite{Tan1983} and Brown \cite{Brown1987}, and for $n=5$ by
Samaris \cite{Samaris2003}.  For general $n$ the best bounds are of the form
$M_n\le 1-\delta$ with a small absolute $\delta>0$
\cite{Ermers1990,Horowitz1978}.  We refer to \cite{MSUV2015} for a survey and
for the variational identities satisfied by extremals.
Several claimed proofs of the full conjecture by Krushkal
\cite{Krushkal2009,Krushkal2016,Krushkal2020,Krushkal2022} have not been
generally accepted \cite{MacDonald2026,MSUV2015}; we do not use them.

\subsection{The Agler--McCarthy reduction}
Every $f\in\Bzero$ can be written $f=e^{-g}$ with $g$ in the Herglotz class
$\HH$ of holomorphic functions with $\Real g\ge0$.  By a theorem going back to
Hummel, Scheinberg and Zalcman \cite{HSZ1977} (see
\cite[Lemma~3.3]{AglerMcCarthy2021} and Section~\ref{sec:prelim} below), if $f$
is $M_n$-extremal then, after multiplication by a unimodular constant,
\begin{equation}\label{eq:atomic}
f(z)=\exp\Bigl(-\sum_{j=1}^N\lambda_j\,\frac{1+\alpha_j z}{1-\alpha_j z}\Bigr),
\end{equation}
for some positive integer $N \le n$, \emph{distinct} points $\alpha_j\in\T$ and
positive masses $\lambda_j$.

We call $N$ the number of \emph{atoms} of $f$, and say that $f$ has \emph{full
degree} if $N=n$.  The extremal functions of Conjecture~\ref{conj:krzyz} have
full degree: $f_n$ has $n$ atoms of mass $1/n$ at the $n$-th roots of $-1$.

Agler and McCarthy \cite{AglerMcCarthy2021} related the Krzy\.z problem to the
following entropy conjecture.

\begin{conjecture}[{\cite[Conj.~1.4]{AglerMcCarthy2021}}]\label{conj:entropy}
Let $p$ be a nonconstant polynomial with all zeros on $\T$ and
$\norm{p}_{L^2(\T)}=1$.  Then
\[
E(p):=\frac1{2\pi}\int_0^{2\pi}\abs{p(e^{i\theta})}^2\log\abs{p(e^{i\theta})}^2
\,d\theta\ \ge\ 1-\log 2,
\]
with equality only for $p(z)=\zeta(\omega+z^n)/\sqrt2$, $\abs\zeta=\abs\omega=1$.
\end{conjecture}

They proved \cite[Thm.~1.8]{AglerMcCarthy2021} that if $f$ is $M_n$-extremal of
full degree and Conjecture~\ref{conj:entropy} holds, then $\abs{a_n(f)}\le
2/e$.  Lei and Zhang \cite[Thm.~1.3]{LeiZhang2026} have recently proved
Conjecture~\ref{conj:entropy}, including the equality statement.  Combining
these (Section~\ref{sec:reduction}): \emph{an $M_n$-extremal function of full
degree is a rotation of $f_n$.}  Hence Conjecture~\ref{conj:krzyz} for a given
$n$ is equivalent to the statement that no $M_n$-extremal function has fewer
than $n$ atoms.  This ``full degree'' step (Step~II in \cite{LeiZhang2026}) is
open in general.  MacDonald \cite{MacDonald2026} proves that an extremal has at
least $cn$ atoms for an absolute constant $c>0$ ($c=0.03$); his subsequent
preprint \cite{MacDonaldConvolution2026} improves this constant to approximately
$0.06$.  The latter also proves $|f(0)|\le e^{-n/N}\le1/e$ and
$\lambda_j\ge1/n$, with equality in $|f(0)|\le1/e$, or in any individual mass
bound, forcing the conjectured extremal.  The mass inequalities also appear
in \cite[Lemmas~2.2--2.3]{Xie2026}.

\subsection{Results}
Our main result settles the full-degree question for $n=6$.

\begin{theorem}\label{thm:main}
Every $M_6$-extremal function has exactly six atoms.
\end{theorem}

\begin{corollary}\label{cor:krzyz6}
    Krzy\.z's conjecture holds for $n=6$: $M_6=2/e$, and the extremal functions
    are exactly $\zeta f_6(\varepsilon z)$, $\abs\zeta=\abs\varepsilon=1$.
\end{corollary}

\begin{theorem}\label{thm:n7}
Every $M_7$-extremal function has exactly seven atoms.  Consequently,
$M_7=2/e$, and the extremal functions are exactly
$\zeta f_7(\varepsilon z)$, $\abs\zeta=\abs\varepsilon=1$.
\end{theorem}

The sixth-coefficient result was first proved in the initial version of
this paper~\cite{WikstromV1}, using a search in all atom coordinates.
Here we focus on mass elimination, which makes the critical-stratum
computation substantially smaller.  Xie's proof of the seventh-coefficient
case~\cite{Xie2026} appeared before our computation for that case was
complete.  Both approaches use the entropy reduction, but their
computational reductions differ: Xie constructs Laurent-polynomial bounds
uniform in the angles and subdivides the mass variables, whereas we
eliminate the masses locally in the critical stratum.  We include the
seventh-coefficient result as a further application of our method.

\subsection{Why the reduction makes the computation feasible}
A direct search over all six-atom configurations would include the equality
functions $f_6$.  Near these functions the gap $2/e-\abs{a_6}$ tends to zero,
so a proof by subdivision and strict upper bounds would need a separate local
argument.  The entropy theorem identifies all full-degree extremals before
the computation begins.  Only $N=1,\dots,5$ remain; after
rotation is removed, the largest of these parameter spaces has dimension
$2N-1=9$, compared with $11$ for six atoms.  These strata stay away from the
equality functions: a limit of measures with at most five atoms has at most
five atoms, so $f_6$ is not in the closure of any of them.  Numerically, the
margin to $2/e$ is about $0.035$ (Section~\ref{sec:landscape}).

The search excludes configurations that could be extremal.  Such a
configuration must satisfy both $\abs{a_6}\ge2/e$ and the first-order
(Karush--Kuhn--Tucker, KKT) conditions for extremality.  On each parameter box we therefore seek either
a strict coefficient bound below $2/e$ or a contradiction to one of these
conditions.  The first-order conditions cannot do this alone: numerical
exploration finds five-atom configurations that satisfy them with
$\abs{a_6}<2/e$ (Section~\ref{sec:landscape}).

The first-order conditions in Theorem~\ref{thm:kkt} are derived from the
classical variational inequalities \cite{HSZ1977,MSUV2015} and the
Fej\'er--Riesz theorem.  For $N=n-1$, rotate the single auxiliary factor to
the positive real axis, retain its real parameter, and solve for the masses
with certified error bounds.  The regular search then has $n$ outer
variables instead of $2n-3$: six instead of nine for $n=6$, and seven
instead of eleven for $n=7$.  Linear infeasibility certificates and interval
contraction handle boxes where the local inverse cannot be verified.
They also avoid any assumed minimum separation of the atoms.

A convexity argument for the Poisson integral bounds the total mass
$t=\sum_j\lambda_j$: Proposition~\ref{prop:poisson} gives
$\abs{a_6}<2/e$ for $t\ge6.989$, improving the general bound $t\le2n$ for
extremals in \cite{MacDonald2026}.  Together with a lower mass bound and
rotation and reflection symmetries, this yields compact search domains.
The smaller strata are treated directly in the atom coordinates.

Sections~\ref{sec:prelim} and~\ref{sec:reduction} give the analytic reduction;
Sections~\ref{sec:kkt} and~\ref{sec:mass} prove the first-order conditions and
mass bounds.  Section~\ref{sec:reduced} develops validated mass elimination,
and Section~\ref{sec:bnb} describes the atom-coordinate search for smaller
strata.  Section~\ref{sec:results} reports the completed computations.
Implementation and reproducibility details are in \cite{Supplement}, and
the floating-point enclosure proofs are in \cite{Enclosures}.
Section~\ref{sec:verification} explains how these parts support the proof.

\subsection*{Tool and computational resource disclosure}
Claude Opus~5.5 (Anthropic) was used as a research assistant for reading the
literature, numerical exploration, the design and implementation of the
branch-and-bound code, and drafting. An independent code review was also done by
OpenAI GPT-6 Astra.  OpenAI Codex was used in developing the reduced solver,
its validation tools, and this revision. The computations were run on Apple M1~Pro and M4~Pro
computers.

The author takes responsibility for the mathematical argument and the
computer-assisted proof, including the generated code.

\section{Preliminaries}\label{sec:prelim}

\subsection{Notation}\label{sec:notation}

For functions $\phi,\psi$ holomorphic near $0$ we write $\phi\eqn\psi$ if their
Taylor coefficients agree up to order $n$.  If $\phi\eqn\psi$,
$\phi(0)=\psi(0)=b$ and $G$ is holomorphic near $b$, then $G\circ\phi\eqn
G\circ\psi$~\cite[Lemma~3.1]{AglerMcCarthy2021}.

Let $\PP_n$ denote the polynomials of degree at most $n$, viewed as a subspace
of the Hardy space $H^2$, and $S=S_n$ the compressed shift $q\mapsto
P_{\PP_n}(zq)$ on $\PP_n$. For $d\in\PP_n$ we write $d(S)=\sum_k \hat d(k)S^k$.
For $p\in\PP_n$, its reflection is $p^*(z)=z^n\overline{p(1/\bar z)}$. When a
different degree is meant, we write $p^{*m}$.  For $\zeta\in\T$ we have
$p^*(\zeta)=\zeta^n \overline{p(\zeta)}$, and $(pq)^{*(m+l)}=p^{*m}q^{*l}$.

\subsection{Herglotz data}

The tuple $(g_0, \dots, g_n)\in\C^{n+1}$ is \emph{solvable Herglotz data} if
there is $g\in\HH$ with $\hat g(k)=g_k$ for $k\le n$.  By the
Carath\'eodory--Toeplitz theorem (see~\cite[Prop.~4.4]{AglerMcCarthy2021}), this
holds if and only if
\[
\mathcal A(g_0, \dots, g_n) := \Real\sum_{k=0}^n g_k S^k \succeq 0 .
\]
So the set $\mathcal C\subset\C^{n+1}$ of solvable data is a closed convex cone,
a \emph{spectrahedron}, and $(1,0,\dots,0)$ is an interior point.

\subsection{Structure of extremals}

If $f=e^{-g}$ is $M_n$-extremal, then $(\hat g(0),\dots, \allowbreak\hat g(n))$
lies on the boundary of $\mathcal C$: otherwise $(\hat g(0),r\hat g(1), \dots,
r^n\hat g(n))$ would be solvable for some $r>1$, and the corresponding function
in $\Bzero$ would have its $n$th coefficient
$r^na_n(f)$~\cite[Lemma~3.3]{AglerMcCarthy2021}. By Pick's theorem, the boundary
data determine $g$ uniquely, and $g$ is of the form
\[
    ia+\sum_{j\le N}\lambda_j\,\frac{1+\alpha_jz}{1-\alpha_jz}
\]
with $N\le n$~\cite[Lemma~3.2]{AglerMcCarthy2021}.  Multiplying $f$ by $e^{ia}$
we obtain \eqref{eq:atomic}.

\subsection{Coefficients}

Put $t=\sum_j\lambda_j$.  Since $f(0)=e^{-t}$, a function of the
form~\eqref{eq:atomic} satisfies
\begin{equation}\label{eq:logf}
    \log f(z) = -t - 2\sum_{k \ge 1}
        \Bigl( \sum_{j} \lambda_j\alpha_j^k \Bigr)z^k .
\end{equation}
For the enclosures, we instead use the product form.  With the Laguerre
generating function $\sum_k L_k^{(-1)}(x)w^k=\exp(-xw/(1-w))$,
\[
    \exp\Bigl(-\lambda\frac{1+w}{1-w}\Bigr) =
    \sum_{k \ge 0} b_k(\lambda) w^k,
\]
where $b_k(\lambda)=e^{-\lambda}L_k^{(-1)}(2\lambda)$.  Here $L_0^{(-1)}=1$ and
\[
L_k^{(-1)}(2\lambda)=\sum_{i=1}^k\binom{k-1}{i-1}\frac{(-2\lambda)^i}{i!}
\qquad(k\ge1).
\]
The function of $w$ is inner, so $\abs{b_k(\lambda)}\le1$.  A function of the
form~\eqref{eq:atomic} is the product of the $N$ factors
$\exp\bigl(-\lambda_j(1+\alpha_jz)/(1-\alpha_jz)\bigr)$, and with $w=\alpha_jz$
the $k$th Taylor coefficient of the $j$th factor is $b_k(\lambda_j)\alpha_j^k$.
Hence the $n$th Taylor coefficient of $f$ is
\begin{equation}\label{eq:product}
    a_n(f)=\sum_{\substack{k_1+\dots+k_N=n\\ k_j\ge0}}\ \prod_{j=1}^N
    b_{k_j}(\lambda_j)\,\alpha_j^{k_j}.
\end{equation}
This holds for all masses $\lambda_j>0$ and points $\alpha_j\in\T$, whether or
not $f$ is extremal.  In Section~\ref{sec:bnb} we regard~\eqref{eq:product} as
a function $a_n(x)$ of the parameters of the configuration.

\subsection{Symmetries}\label{sec:sym}

The maps $f\mapsto f(\varepsilon z)$ ($\varepsilon\in\T$) and $f\mapsto
\overline{f(\bar z)}$ preserve $\Bzero$ and $\abs{a_n}$.  On configurations
they act by $\alpha_j\mapsto\varepsilon\alpha_j$ and $\alpha_j\mapsto
\bar\alpha_j$, with the masses unchanged.  Relabelling the atoms does not
change $f$.

\section{Full-degree extremals}\label{sec:reduction}

\begin{theorem}[Agler--McCarthy, Lei--Zhang]\label{thm:am-lz}
Let $f$ be $M_n$-extremal with exactly $n$ atoms.  Then $f=\zeta
f_n(\varepsilon z)$ with $\abs\zeta=\abs\varepsilon=1$.  In particular
$M_n=2/e$.
\end{theorem}

\begin{proof}[Sketch]
After a rotation and multiplication by a unimodular constant, $g=-\log f$ is
\emph{normalized} in the sense of~\cite[Def.~6.1]{AglerMcCarthy2021}, and
$\eta=a_n(f)>0$~\cite[Lemma~6.3]{AglerMcCarthy2021}.  Let
$\gamma=\nu^{-1}\prod_j(z-\tau_j)$ be the unit vector spanning $\ker\Real g(S)$
(Lemma~4.7 there).  An extremal is a local maximum of $\abs{a_n}^2$ on the set
of full-degree data, hence a critical point (Prop.~4.6 there).  The critical
point equation (Thm.~5.1 and Prop.~6.4 there) gives
\[
-g\eqn\log\frac{\eta}{\nu^2}+\log\prod_{j=1}^n(1-\bar\tau_jz)^2 .
\]
Comparing Toeplitz matrices, $-\Real g(S)$ is the Toeplitz operator on $\PP_n$
with symbol $\log\eta+\log\abs\gamma^2$ \cite[\S7]{AglerMcCarthy2021}.  Pairing
with $\gamma\in\ker\Real g(S)$ gives $E(\gamma)=-\log\eta$.  The polynomial
$\gamma$ has degree $n$ and simple zeros on $\T$, so the entropy inequality
of Lei and Zhang \cite[Thm.~1.3]{LeiZhang2026} (Conjecture~\ref{conj:entropy})
gives $\eta\le 2/e$.  Since $\eta=M_n\ge 2/e$, equality holds.  By the equality
statement of \cite[Thm.~1.3]{LeiZhang2026}, $\gamma=\zeta'(\omega+z^n)/\sqrt2$, so the $\tau_j$ are the $n$-th
roots of $-\omega$.  Then \cite[Prop.~6.6]{AglerMcCarthy2021} reads $\sum_j
\lambda_j\tau_j^{-k}=\frac1k\sum_j\tau_j^{-k}$ for $1\le k\le n$, i.e.
$\sum_j\lambda_j\tau_j^{-k}=0$ for $k<n$ and $\sum_j\lambda_j=1$.  By the
Vandermonde determinant, $\lambda_j=1/n$ for all $j$, and $f$ is a rotation of
$f_n$.
\end{proof}

\section{First-order conditions}\label{sec:kkt}

For $f\in\Bzero$ with Taylor coefficients $a_k=a_k(f)$ and $a_n\ne0$
(as for an extremal), put
\[
P(\zeta)=a_n+2\sum_{k=1}^na_{n-k}\zeta^k,\qquad \omega=\frac{a_n}{\abs{a_n}},\qquad
h(\varphi)=\Real\bigl(\bar\omega P(e^{i\varphi})\bigr).
\]
Thus $-P(\zeta)$ is the first variation of $a_n$ when a small atom is added at
$\zeta\in\T$: $a_n\bigl(f\exp(-\varepsilon\frac{1+\zeta z}{1-\zeta z})\bigr)=
a_n-\varepsilon P(\zeta)+O(\varepsilon^2)$ (see Lemma~\ref{lem:var}).

\begin{lemma}[{\cite[p.~173]{HSZ1977}, \cite[Prop.~3(d),(e)]{MSUV2015}}]\label{lem:var}
Let $f$ be $M_n$-extremal of the form~\eqref{eq:atomic}.  Then $h\ge0$ on
$\R$, and $h(\theta_j)=0$ for $j=1,\dots,N$, where $\alpha_j=e^{i\theta_j}$.
\end{lemma}

\begin{proof}
For $\zeta\in\T$ and $\varepsilon>0$ the function
$f_\varepsilon=f\exp(-\varepsilon\frac{1+\zeta z}{1-\zeta z})$ belongs to
$\Bzero$.  Since $\exp(-\varepsilon\frac{1+\zeta z}{1-\zeta z})=1-\varepsilon
(1+2\sum_{k\ge1}\zeta^kz^k)+O(\varepsilon^2)$, we get
$a_n(f_\varepsilon)=a_n-\varepsilon P(\zeta)+O(\varepsilon^2)$, hence
$\abs{a_n(f_\varepsilon)}=\abs{a_n}-\varepsilon\Real(\bar\omega P(\zeta))+
O(\varepsilon^2)$.  Extremality gives $h\ge0$.  If $\zeta=\alpha_j$, then
$f_\varepsilon\in\Bzero$ also for $-\lambda_j<\varepsilon<0$, which gives
$h(\theta_j)\le0$.
\end{proof}

(MSUV normalize $a_n>0$, so that $\omega=1$ and $h=\Real P$ on $\T$.  They
state (d) on the closed disk, and they use the Fej\'er--Riesz representation
$\Real P=\abs Q^2$ in the proof of their Theorem~1.)

The following factorization makes the variational conditions explicit for
any number of atoms.  For $N=n$ it recovers the critical point equation of
Agler and McCarthy \cite[Thm.~5.1, Props.~6.4 and 6.6]{AglerMcCarthy2021}.

\begin{theorem}\label{thm:kkt}
Let $f$ be an $M_n$-extremal function of the form~\eqref{eq:atomic} with
$N\le n$ atoms, $t=\sum_j\lambda_j$, and put $m=n-N$.  Then there are
$\sigma_1,\dots,\sigma_m\in\C$ with $0<\abs{\sigma_i}\le 1$ such that:
\begin{enumerate}
\item[(a)] for all real $\varphi$,
\begin{equation}\label{eq:kkt-h}
h(\varphi)=\frac{e^{-t}}{\prod_i\abs{\sigma_i}}\ \prod_{j=1}^N\abs{e^{i\varphi}-\alpha_j}^2
\ \prod_{i=1}^m\abs{1-\sigma_ie^{-i\varphi}}^2 ;
\end{equation}
\item[(b)] for $k=1,\dots,n$,
\begin{equation}\label{eq:kkt-power}
\sum_{j=1}^N\Bigl(\lambda_j-\frac1k\Bigr)\alpha_j^k
=\frac1{2k}\sum_{i=1}^m\bigl(\sigma_i^k+\bar\sigma_i^{-k}\bigr) ;
\end{equation}
\item[(c)] the phase of $a_n$ is
\begin{equation}\label{eq:kkt-phase}
\omega=\prod_{j=1}^N(-\alpha_j)\ \prod_{i=1}^m\frac{-\sigma_i}{\abs{\sigma_i}} .
\end{equation}
\end{enumerate}
In particular, if $N=n-1$, then $X:=\sum_j(\lambda_j-1)\alpha_j$ satisfies
$\abs X\ge1$,
\begin{equation}\label{eq:kkt-cheb}
    \begin{split}
        k\sum_{j=1}^{n-1}\Bigl(\lambda_j-\frac1k\Bigr)\alpha_j^k
        &= T_k(\abs X)\Bigl(\frac{X}{\abs X}\Bigr)^{k}\quad(k=2,\dots,n),
        \qquad\text{and}
        \\
        \omega &= (-1)^n\alpha_1\cdots\alpha_{n-1}\frac{X}{\abs X},
    \end{split}
\end{equation}
where $T_k$ is the Chebyshev polynomial of degree $k$.
\end{theorem}

\begin{proof}[Proof of Theorem~\ref{thm:kkt}]
\emph{Step 1: factorization of $h$.}  $h$ is a real trigonometric polynomial
of degree at most $n$.  It is nonnegative, so each of its zeros $\theta_j$ has
even order.  Hence $h(\varphi)=\prod_j\abs{e^{i\varphi}-\alpha_j}^2\,w(\varphi)$,
where $w$ is a nonnegative trigonometric polynomial of degree at most $m$.
By the Fej\'er--Riesz theorem there is a polynomial $v$ of degree at most
$m$, with no zeros in $\D$, such that $w=\abs{v(e^{i\varphi})}^2$.  Thus
$h=\abs{q(e^{i\varphi})}^2$ with $q(\zeta)=\prod_j(\zeta-\alpha_j)\,v(\zeta)\in
\PP_n$.

\emph{Step 2: coefficient matching.}  On $\T$, $\bar\omega P(\zeta)=\bar\omega
a_n+2\sum_{k\ge1}\bar\omega a_{n-k}\zeta^k$, so the Fourier coefficients of $h$
are $\hat h(0)=\Real(\bar\omega a_n)=\abs{a_n}= \bar\omega a_n$ and $\hat
h(k)=\bar\omega a_{n-k}$ for $1\le k\le n$.  On the other hand, $\hat
h(k)=\sum_iq_{i+k}\bar q_i$ for $h=\abs q^2$.  Let $r(z)=\overline{q(\bar z)}$
be the polynomial with conjugated coefficients. Then
$\widehat{rr^*}(n-k)=\sum_i\bar q_iq_{i+k}=\hat h(k)$, where $r^*$ is the
reflection in degree $n$.  Therefore
\begin{equation}\label{eq:kkt-product}
    f \eqn \omega\,r\,r^*=\omega\,c_\alpha\prod_{j=1}^N(1-\alpha_jz)^2\,u\,u^{*m},
\end{equation}
where $c_\alpha=\prod_{j=1}^N(-\bar\alpha_j)$ and $u(z)=\overline{v(\bar z)}$,
since $r(z)=\prod_j(z-\bar\alpha_j)\,u(z)=c_\alpha\prod_j(1-\alpha_jz)\,u(z)$
and $r^*(z)=\prod_j(1-\alpha_jz)\,u^{*m}(z)$.

\emph{Step 3: the atoms of $u$.}  $u$ has no zeros in $\D$.  Evaluating
\eqref{eq:kkt-product} at $0$, we get
$e^{-t}=\omega\,c_\alpha\,u(0)\,u^{*m}(0)\ne0$. So $u(0)\ne0$ and
$u^{*m}(0)\ne0$, i.e.\ $u$ has exact degree $m$.  Write
$u(z)=u(0)\prod_{i=1}^m(1-\sigma_iz)$.  Then $0<\abs{\sigma_i}\le1$ and
$u^{*m}(z)=\overline{u(0)}\prod_i(z-\bar\sigma_i)$, so
\[
e^{-t}=\omega\prod_j(-\bar\alpha_j)\,\abs{u(0)}^2\prod_i(-\bar\sigma_i) .
\]
The phase gives~\eqref{eq:kkt-phase}, while the modulus gives
$\abs{u(0)}^2=e^{-t}/\prod_i\abs{\sigma_i}$.  Since
$\abs{v(e^{i\varphi})}=\abs{u(e^{-i\varphi})}=\abs{u(0)}\prod_i
\abs{1-\sigma_ie^{-i\varphi}}$, this is~\eqref{eq:kkt-h}.

\emph{Step 4: logarithms.}  Taking logarithms in~\eqref{eq:kkt-product}
(Section~\ref{sec:notation}), with $u^{*m}(z)=c'\prod_i(1-z/\bar\sigma_i)$, and comparing the
coefficients of $z^k$, $1\le k\le n$, using~\eqref{eq:logf} and
$\log(1-wz)=-\sum_kw^kz^k/k$, gives
\[
-2\sum_j\lambda_j\alpha_j^k=-\frac2k\sum_j\alpha_j^k-\frac1k\sum_i
\bigl(\sigma_i^k+\bar\sigma_i^{-k}\bigr),
\]
which is~\eqref{eq:kkt-power}.

\emph{Step 5: $N=n-1$.}  Write $\sigma_1=e^{-s+i\psi}$ with $s\ge0$.  The case
$k=1$ of~\eqref{eq:kkt-power} gives $X=\frac12(\sigma_1+\bar\sigma_1^{-1})=
\cosh(s)e^{i\psi}$, so $\abs X=\cosh s\ge1$ and $X/\abs X=\sigma_1/\abs{\sigma_1}$.
For $k\ge2$ the right side of~\eqref{eq:kkt-power}, multiplied by $k$, is
$\cosh(ks)e^{ik\psi}=T_k(\abs X)(X/\abs X)^k$.  The phase relation follows from
\eqref{eq:kkt-phase} with $N=n-1$.
\end{proof}

For $N=n-1$ the relations~\eqref{eq:kkt-cheb} are $n-1$ complex equations and
one real equation between the $2n-3$ real parameters of the configuration
modulo rotation.  The search in~\cite{WikstromV1} used these relations
to exclude boxes.  Section~\ref{sec:reduced} instead retains the auxiliary
parameter and exploits linearity in the masses.

\begin{remark}\label{rem:kkt}
\begin{enumerate}
\item \emph{Full degree.}  For $N=n$ we have $m=0$.  Then
\eqref{eq:kkt-product} is the critical point equation
\cite[Thm.~5.1]{AglerMcCarthy2021}, and~\eqref{eq:kkt-power} reads
$\sum_j\lambda_j\alpha_j^k=\frac1k\sum_j\alpha_j^k$ \cite[Prop.~6.6]{AglerMcCarthy2021}.
Equation~\eqref{eq:kkt-phase} gives $\omega=\prod_j(-\alpha_j)$,
which is the normalization in \cite[Def.~6.1, Lemma~6.3]{AglerMcCarthy2021}.
Finally~\eqref{eq:kkt-h} becomes $h(\varphi)=e^{-t}\prod_j\abs{e^{i\varphi}-
\alpha_j}^2$; for $f_n$ this is $(2/e)(1+\cos n\varphi)$.
\item \emph{Completeness.}  The conditions of Lemma~\ref{lem:var} are the
complete first-order conditions for maximizing $\abs{a_n}$ over the measures.
Let $\mu=\sum_j\lambda_j\delta_{\alpha_j}$ be the measure of $f$, and let $\nu$
be any finite positive measure on $\T$.  The Herglotz data depend linearly on
the measure, and $a_n$ depends analytically on these data, so the computation in the
proof of Lemma~\ref{lem:var} shows that the right derivative of $\abs{a_n}$ at
$s=0$ along $\mu+s(\nu-\mu)$ is
\[
-\int h\,d\nu+\int h\,d\mu ,
\]
where $h(\varphi)$ is evaluated at $e^{i\varphi}$.  Adding an imaginary
constant to $g$ only changes the unimodular factor, and does not change
$\abs{a_n}$.  If $h\ge0$ and $h=0$ on the support of $\mu$, the derivative is
$-\int h\,d\nu\le0$ for every $\nu$; conversely, $\nu=\mu+\delta_\zeta$ gives
$h\ge0$, and $\nu=\mu-\lambda_j\delta_{\alpha_j}$ gives $h(\theta_j)\le0$.  So
Lemma~\ref{lem:var} says precisely that $\abs{a_n}$ does not increase to first
order in any feasible direction, and by continuity in any direction of the
tangent cone, the closure of the cone generated by the $\nu-\mu$.  Steps~2--4 can be
reversed, so~\eqref{eq:kkt-power} and~\eqref{eq:kkt-phase} together are
equivalent to Lemma~\ref{lem:var}.  In particular they contain all first-order
information.  The phase should be accounted for when reversing the
logarithmic calculation; for $N=n-1$ it already follows from the full
moment system, as shown in Proposition~\ref{prop:central-coeff} below.
Equivalently, one can argue with
the normal cone of the spectrahedron $\mathcal C$ as in
\cite{AglerMcCarthy2021}, see \cite[Thm.~23.9]{Rockafellar1970}.
\item \emph{A priori bound.}  Taking moduli in~\eqref{eq:kkt-cheb} gives
$T_k(\abs X)=k\abs{Y_k}$ with $Y_k=\sum_j(\lambda_j-1/k)\alpha_j^k$.  Since
$\abs{Y_k}\le t+N/k$, for $n=6$ and $t\le7$ we get $\cosh 6s=T_6(\abs X)\le
6t+5\le47$ in the notation of Step~5.  Hence $1\le\abs X\le1.301$ and
$\abs{\sigma_1}=e^{-s}\ge94^{-1/6}>0.468$.  By~\eqref{eq:kkt-h}, $h$ has a zero on $\T$
besides the atoms iff $\abs{\sigma_1}=1$.
\end{enumerate}
\end{remark}

\section{Bounds for the total mass}\label{sec:mass}

\begin{lemma}\label{lem:tlo}
If $\phi\colon\D\to\D$ is holomorphic, then $\abs{a_n(\phi)}\le1-\abs{a_0(\phi)}^2$
for $n\ge1$.  Consequently, if $f\in\Bzero$ has $\abs{f(0)}=e^{-t}$ with
$t<T_{\mathrm{lo}}:=-\frac12\log(1-2/e)=0.665446\ldots$, then $\abs{a_n(f)}<2/e$.
\end{lemma}

\begin{proof}
The average $\frac1n\sum_{\ell=0}^{n-1}\phi(e^{2\pi i\ell/n}z)$ maps $\D$ into
$\D$ and equals $\psi(z^n)$ for a holomorphic $\psi\colon\D\to\D$ with
$\psi(0)=a_0$ and $\psi'(0)=a_n$.  Apply the Schwarz--Pick lemma to $\psi$.  This is F.~Wiener's argument,
see \cite[\S1.7]{MSUV2015}.
\end{proof}

An upper bound for the mass follows from a convexity argument that applies
to every function in $\Bzero$, without any assumption on its atoms.

\begin{proposition}\label{prop:poisson}
Let $f\in\Bzero$ with $\abs{f(0)}=e^{-t}$.  For every $r\in(0,1)$,
\[
\abs{a_n(f)}\le r^{-n}\,\frac1{2\pi}\int_0^{2\pi}\exp\bigl(-t\,P_r(\theta)\bigr)\,d\theta,
\]
where $P_r(\theta)=\frac{1-r^2}{1-2r\cos\theta+r^2}$ is the Poisson kernel.
In particular, $\abs{a_6(f)}<2/e$ whenever $t\ge 6.989$.
\end{proposition}

\begin{proof}[Proof of Proposition~\ref{prop:poisson}]
Write $f=\zeta\exp(-\int_\T\frac{\xi+z}{\xi-z}\,d\mu(\xi))$ with $\mu\ge0$ of
mass $t$.  Then $\abs{f(re^{i\theta})}=\exp(-\int P_r(\theta-\arg\xi)\,
d\mu(\xi))$.  By Cauchy's estimate,
\[
\abs{a_n}\le r^{-n}\frac1{2\pi}\int_0^{2\pi}\abs{f(re^{i\theta})}\,d\theta .
\]
With $\nu=\mu/t$, Jensen's inequality for the probability measure $\nu$ gives
\[
\exp\Bigl(-\int tP_r(\theta-\arg\xi)\,d\nu(\xi)\Bigr)
\le\int\exp\bigl(-tP_r(\theta-\arg\xi)\bigr)\,d\nu(\xi).
\]
Integrate over $\theta$ and use Fubini: the inner integral does not depend on
$\xi$.  The right-hand side of the proposition is decreasing in $t$.  For $n=6$,
$r=0.9161$ and $t=6.989$ it is at most $0.735487<2/e=0.735759$; this was
certified in ball arithmetic with rigorous Riemann sums, using that $P_r$ is
even and decreasing on $[0,\pi]$.
\end{proof}

\begin{table}[ht]
\centering
\begin{tabular}{ccccc}
\toprule
$n$ & $T_{\mathrm{hi}}(n)$ used & certified at $t$ & $r$ & certified upper bound \\
\midrule
3 & 3.54 & 3.539 & 0.8456 & 0.734882 \\
4 & 4.69 & 4.689 & 0.8790 & 0.734844 \\
5 & 5.84 & 5.839 & 0.9009 & 0.735144 \\
6 & 6.99 & 6.989 & 0.9161 & 0.735487 \\
7 & 8.15 & 8.149 & 0.9280 & 0.735332 \\
\bottomrule
\end{tabular}
\medskip
\caption{Mass cutoffs from Proposition~\ref{prop:poisson}: for $t\ge
T_{\mathrm{hi}}(n)-0.001$ every $f\in\Bzero$ has $\abs{a_n}<2/e=0.7357589$.
The margin of $0.001$ covers the rounding of the decimal cutoff to a double.}
\label{tab:thi}
\end{table}

For $n=7$, the last row is certified by the same monotone Riemann-sum
argument, with the displayed bound rounded upward.  Details of the cutoff
script are recorded in the supplement.

\begin{remark}
The cutoff $T_{\mathrm{hi}}(n)$ grows roughly like $1.17n$ for $3\le n\le6$,
with $r$ optimized numerically for each $n$.  In the search a box-wise version of the same bound is
used (test P below), with the actual atom positions.
\end{remark}

\section{Validated mass elimination}\label{sec:reduced}

The first-order conditions are linear in the masses.  For the critical
stratum $N=n-1$, exploiting this structure reduces the regular search from
$2n-3$ to $n$ outer variables.  We give the reduction and its correctness
argument here; the enclosure formulas are in
\cite[Section~S9]{Supplement}.

Write the auxiliary parameter in Theorem~\ref{thm:kkt} as
$\sigma=r e^{i\psi}$, rotate all atoms by $e^{-i\psi}$, and put
$q=(r+r^{-1})/2\ge1$.  Since $X=q e^{i\psi}$ before rotation, this loses no
candidate.  Rotation has been used, so no atom angle is fixed to zero.
With $\alpha_j=e^{i\theta_j}$ the necessary equations are
\begin{equation}\label{eq:reduced-moments}
 \sum_{j=1}^{n-1}(\lambda_j-1/k)e^{ik\theta_j}=T_k(q)/k,
 \qquad k=1,\dots,n.
\end{equation}
Splitting real and imaginary parts gives $A(\theta)\lambda=b(\theta,q)$,
where column $j$ of $A$ is
$(\cos\theta_j,\sin\theta_j,\dots,\cos n\theta_j,\sin n\theta_j)^T$.
The outer parameters are $x=(\theta_1,\dots,\theta_{n-1},q)$.

\begin{proposition}\label{prop:mass-elimination}
For distinct atoms, $A$ has full real column rank.  On a parameter box $B$,
let $R$ be a constant $(n-1)\times2n$ matrix.  If
\begin{equation}\label{eq:inverse-cert}
 \sup_{x\in B}\|I-RA(x)\|_w<1
\end{equation}
for a weighted infinity norm, every solution of the moment equations lies
on the graph $\lambda_R(x)=(RA(x))^{-1}Rb(x)$.
\end{proposition}
\begin{proof}
A real vector in the kernel of $A$ lies in the kernel of
$(\alpha_j^k)_{1\le k,j\le n-1}$, an invertible Vandermonde matrix times
$\operatorname{diag}(\alpha_j)$.  Thus $A$ has full column rank.
Condition~\eqref{eq:inverse-cert} proves invertibility of $RA(x)$ by the
Neumann series, and the graph formula follows by multiplying by $R$.
\end{proof}

At a box center, numerical linear algebra proposes $R$ and an affine
approximation to $\lambda_R$.  A verified inverse bound and a second-order
residual enclosure give a rigorous error for this approximation.  Summing
the affine expressions before taking ranges preserves dependence in the
total mass.  The remaining moment equations, nonnegative masses, and the
scalar mass bounds then provide exclusion tests.  Numerical inverses and
rank decisions are only proposals; failure to verify
\eqref{eq:inverse-cert} leaves the box unresolved.

The moment equations also give a simpler coefficient formula.
\begin{proposition}\label{prop:central-coeff}
Suppose~\eqref{eq:reduced-moments} holds and $q\ge1$.  Write
$p(z)=\prod_j(z-\alpha_j)=\sum_{k=0}^{n-1}p_kz^k$.  Then
\begin{equation}\label{eq:central-coeff}
 |a_n|=e^{-t}\left(2q\sum_{k=0}^{n-1}|p_k|^2
                  -2\Real\sum_{k=0}^{n-2}p_{k+1}\overline{p_k}\right),
\end{equation}
and $a_n/|a_n|=(-1)^n\prod_j\alpha_j$.
\end{proposition}
\begin{proof}
Taking formal exponentials of the moment equations gives
$f(z)\eqn e^{-t}Q(z)$, where
$Q(z)=\prod_j(1-\alpha_jz)^2(1-2qz+z^2)$.  On the circle,
\[
 Q(e^{-i\varphi})=(-1)^n\Bigl(\prod_j\alpha_j\Bigr)e^{-in\varphi}
                 |p(e^{i\varphi})|^2(2q-2\cos\varphi).
\]
The central coefficient is the phase factor times the mean of the last
two factors.  This mean is strictly positive and equals the expression in
parentheses in~\eqref{eq:central-coeff}.
\end{proof}

To cover singular boxes, retain the mass intervals and form a necessary
linear relaxation $H\lambda\le h$, $\lambda\ge0$, $\sum\lambda_j\le T$,
where $T=T_{\mathrm{hi}}(n)$.  For any $y\ge0$, feasibility implies
\begin{equation}\label{eq:dual-cert}
 T\min(0,\min_j(H^Ty)_j)\le y^Th.
\end{equation}
Indeed, with $s=\min(0,\min_j(H^Ty)_j)$,
$Ts\le s\sum\lambda_j\le y^TH\lambda\le y^Th$.
Linear programming proposes $y$; enclosing arithmetic verifies a strict
violation of~\eqref{eq:dual-cert}.  Interval contraction and a direct atomic
coefficient bound supplement this rank-free test.

Order the angles and use conjugation to impose
$\theta_1+\theta_{n-1}\le2\pi$.  The last moment equation gives
$T_n(q)\le nT_{\mathrm{hi}}(n)+n-1$; in particular $q<1.3$ for $n=7$.
These inequalities define a compact root with outward-rounded endpoints.
It includes $q=1$, coinciding angles, zero masses, and the periodic seam.
Every contraction preserves a possible extremal, every successful test
excludes one, and the two children of a bisection cover the contracted
parent.  Thus an empty final frontier excludes the entire critical stratum.
A failed chart is never an exclusion.  Masses can also be subdivided in the
fallback, so $n$ is the regular outer dimension, not a worst-case bound.

The labeled moment equations must be retained at collisions: their $1/k$
terms count the original factors, even for a zero mass.  A smaller-stratum
extremal has more auxiliary factors in Theorem~\ref{thm:kkt} and need not
satisfy these limiting equations.  Covering the boundary of this system
therefore does not replace the separate smaller-stratum computations.

\section{Branch-and-bound in the atom coordinates}\label{sec:bnb}

For the smaller strata $N\le n-2$, we use a search in the masses and atom
angles.  We describe its compact parameter domains and exclusion tests.  Proposition~\ref{prop:correct} connects
the outcome of the search to the full-degree property.

\subsection{Reduction to a compact domain}
Fix $n$ and $1\le N\le n-1$.  A configuration of $N$ atoms consists of masses
$\lambda_j$ and points $\alpha_j\in\T$.  After a rotation
(Section~\ref{sec:sym}) we may take $\alpha_1=1$, and we parametrize
configurations by
\[
x=(\lambda_1,\dots,\lambda_N,\theta_2,\dots,\theta_N)\in\R^{2N-1},\qquad
\alpha_1=1,\ \alpha_j=e^{i\theta_j}.
\]
As in Section~\ref{sec:mass}, $t=\sum_j\lambda_j$ denotes the total mass.  The
symmetries of Section~\ref{sec:sym} let us assume that the atom at $1$ is a
heaviest one, that the other atoms are ordered by their arguments, and, after
a reflection, that $\theta_2+\theta_N\le2\pi$.  By Lemma~\ref{lem:tlo} and
Proposition~\ref{prop:poisson}, an extremal has $T_{\mathrm{lo}}\le t\le
T_{\mathrm{hi}}(n)$, with the cutoffs of Table~\ref{tab:thi}.  Accordingly, let
$D_N$ be the set of $x$ with
\begin{equation}\label{eq:fd}
\begin{gathered}
\lambda_j\ge0\ (j\ge1),\quad\lambda_1\ge\lambda_j\ (j\ge2),\quad
T_{\mathrm{lo}}\le t\le T_{\mathrm{hi}}(n),\\
0\le\theta_2\le\dots\le\theta_N\le2\pi,\quad \theta_2+\theta_N\le2\pi,
\end{gathered}
\end{equation}
where the conditions on the angles are void for $N=1$.  We allow vanishing
masses, so that $D_N$ is closed.  For $N\ge2$, $D_N$ is
contained in the root box
$R_N=[0,T]\times[0,T/2]^{N-1}\times[0,\pi]\times[0,2\pi]^{N-2}$,
$T=T_{\mathrm{hi}}(n)$, because $0\le\lambda_j\le\min(\lambda_1,t-\lambda_1)\le
t/2$ for $j\ge2$ and $2\theta_2\le\theta_2+\theta_N\le2\pi$; for $N=1$ we put
$R_1=[0,T]$.

\begin{lemma}\label{lem:fd}
If there is an $M_n$-extremal function with exactly $N$ atoms, $N<n$, then
there is one whose parameter $x$ lies in $D_N$ and satisfies
$\abs{a_n(x)}\ge2/e$.
\end{lemma}

\begin{proof}
Rotate so that a heaviest atom sits at $1$, and label the other atoms by
increasing argument in $(0,2\pi)$.  If $\theta_2+\theta_N>2\pi$, apply the
reflection $\alpha\mapsto\bar\alpha$ and relabel.  This keeps the heaviest atom
at $1$ and replaces $(\theta_2,\theta_N)$ by $(2\pi-\theta_N,2\pi-\theta_2)$.
The new function is again $M_n$-extremal, with $\abs{a_n}=M_n\ge 2/e$, so
Lemma~\ref{lem:tlo} and Proposition~\ref{prop:poisson} apply to it.
\end{proof}

\subsection{The tests}
The computation proves that no $x\in D_N$ satisfies the conditions of
Lemma~\ref{lem:fd}.  It maintains a set of boxes covering $R_N$ and bisects
boxes until each is \emph{closed} by one of the following tests.  Each test is
valid on the whole closed box, including degenerate points with a vanishing
mass or coinciding angles.
\begin{description}
\item[T] $\sum_j\sup\lambda_j<T_{\mathrm{lo}}$ or $\sum_j\inf\lambda_j>
T_{\mathrm{hi}}(n)$.
\item[S] The box misses $D_N$: $\sup\lambda_1<\inf\lambda_j$ for some $j$, or
$\inf\theta_j>\sup\theta_{j+1}$ for some $j$, or $\inf\theta_2+\inf\theta_N>2\pi$.
\item[P] A box-wise version of Proposition~\ref{prop:poisson}: for some
$r\in(0,1)$, an upper Riemann sum for
$r^{-n}\frac1{2\pi}\int\exp(-\sum_j\lambda_jP_r(\theta-\theta_j))\,d\theta$ is
$<2/e$.  On each cell of the Riemann sum, $\lambda_j$ is replaced by its lower
bound on the box, and $P_r(\theta-\theta_j)$ by its minimum over the cell and
the $\theta_j$-interval.
\item[V] $\sup\abs{a_n}<2/e$ on the box.
\end{description}
The tests are applied in this order.  A box that reaches a maximal depth, or a
minimal width, without being closed is reported as undecided; no box was
undecided in any of the runs.

\begin{proposition}\label{prop:correct}
If for some $N<n$ the search closes every box of a cover of $R_N$ with the tests
T, S, P, V, then no $M_n$-extremal function has exactly $N$ atoms.
\end{proposition}

\begin{proof}
Let $x$ be as in Lemma~\ref{lem:fd}, and let $B$ be a box of the cover
containing $x$.  Then $x\in D_N$, so $B$ was not closed by S.  Also $T_{\mathrm{lo}}\le
t(x)\le T_{\mathrm{hi}}$, so it was not closed by T.  $\abs{a_n(x)}\ge2/e$
excludes P and V (test P bounds the Cauchy integral in the proof of
Proposition~\ref{prop:poisson}, which is at least $\abs{a_n(x)}$ at every
point of $B$).  So no test closes $B$, a contradiction.
\end{proof}

\subsection{Enclosures}\label{sec:tm}
Test V needs an upper bound for $\abs{a_n}$ over a box.  On a box with centre $c$ and radii $\rho$
we write $x=c+\rho\odot\delta$ with $\delta\in[-1,1]^{2N-1}$, and represent
each quantity by a complex \emph{Taylor model} of degree two
\cite{MakinoBerz2003}: a quadratic polynomial $p(\delta)$ together with a bound
$\epsilon$ for $\sup_\delta\abs{\phi-p}$.  The polynomial parts are added
coefficientwise and multiplied as polynomials.  Products are truncated at
degree two, with the discarded terms bounded and included in the remainder.

For each atom, $b_k(\lambda_j)$ and $\alpha_j^k=e^{ik\theta_j}$ are replaced
by their second-order Taylor polynomials at the centre of the box, with
remainders at most $C_k\rho^3/6$ and $(k\rho')^3/6$, where $\rho$ and $\rho'$
are the radii of $\lambda_j$ and $\theta_j$ and $C_k\ge\sup\abs{b_k'''}$ are
certified constants.  The Taylor model of $a_n$ is then obtained by multiplying
out~\eqref{eq:product}; the recursion that follows from~\eqref{eq:logf} cancels
heavily and gives much worse bounds.  Finally, the modulus of a Taylor model is bounded from its
coefficients, with a verified eigenvalue bound \cite{Rump2006} for the
quadratic part when that is sharper than the entrywise bound.

All arithmetic is IEEE~754 double precision with upward rounding, which makes
upper bounds of monotone expressions valid automatically; lower bounds are
computed as negated upper bounds, and the rounding errors of the Taylor-model
coefficients are bounded a priori and added to the remainder.  The
exponential, sine and cosine are implemented from the basic operations, since
the system library gives no error bounds in directed rounding.  The constants
($C_k$, $T_{\mathrm{lo}}$, $T_{\mathrm{hi}}(n)$, $2/e$) are certified in ball
arithmetic \cite{Arb}.  These fixed-precision enclosures avoid the cost of arbitrary-precision
arithmetic in the search loop.
The justification of these bounds and the checks on the implementation are
discussed in Section~\ref{sec:verification}.

\subsection{Search}
A box is bisected at the midpoint of the coordinate that contributes most to
the Taylor model of $a_n$.  The search is parallel: the work is split into
\emph{jobs}, boxes that are explored depth first by one thread and split
further when other threads are idle.  Every job is recorded in a log, together
with a hash of the source code, and an interrupted run can be resumed from its
log.  The job records and leaf counts allow a separate structural audit,
described in Section~\ref{sec:verification}; the implementation details are in
\cite[Section~S5]{Supplement}.

\section{Results}\label{sec:results}

\subsection{The case \texorpdfstring{$n=6$}{n=6}}\label{sec:reduced-results}
Table~\ref{tab:n6} records completed searches for all five strata.  The
critical stratum $N=5$ uses mass elimination and takes $492.793$\,s
(8.2 minutes), or $1.559$ CPU hours, on an Apple M4 Pro with 12 workers.
The smaller strata use the atom-coordinate search; their archived runs on
the M4 Pro took about 6 minutes in total.  All searches finish without
undecided boxes, and all structural audits pass.  The timings refer to the
separate completed runs and exclude auditing.

\begin{table}[ht]
\centering
\small
\begin{tabular}{c l r r}
\toprule
$N$ & method & nodes & leaves\\
\midrule
1 & atom coordinates & 31 & 16\\
2 & atom coordinates & 6\,849 & 3\,425\\
3 & atom coordinates & 1\,168\,309 & 584\,155\\
4 & atom coordinates & 289\,078\,879 & 144\,539\,440\\
5 & mass elimination & 36\,317\,049 & 18\,158\,525\\
\bottomrule
\end{tabular}
\caption{Completed exclusions for $n=6$.  The first four rows retain the
smaller-stratum runs of~\cite{WikstromV1}; the last uses the reduced search.
In each row, $\text{nodes}=2\,\text{leaves}-1$.}
\label{tab:n6}
\end{table}

For comparison, the critical-stratum search in the first arXiv version
\cite{WikstromV1} used $27\,957\,495\,167$ nodes and $86\,072$\,s
(about 24 hours) on the M4 Pro with 12 workers.  Mass elimination reduces
the node count by a factor of about $770$ and wall time by about $175$.
Its nodes are more expensive, but the search is much smaller.

\begin{proof}[Proof of Theorem~\ref{thm:main}]
An $M_6$-extremal has at most six atoms.  For $N=1,\dots,4$, the completed
covers in Table~\ref{tab:n6} and Proposition~\ref{prop:correct} exclude an
extremal.  For $N=5$, every extremal has a normalized representative in the
compact root of Section~\ref{sec:reduced}; the completed reduced search and
the exclusion and coverage argument there rule out that representative.
Thus an extremal has exactly six atoms.
\end{proof}

\begin{proof}[Proof of Corollary~\ref{cor:krzyz6}]
Combine Theorem~\ref{thm:main} with Theorem~\ref{thm:am-lz} for $n=6$.
\end{proof}

As controls, searches also close every stratum for the known cases
$n=3,4,5$; see \cite[Table~S1]{Supplement}.  The reduced
$n=5$, $N=4$ search has $927\,969$ nodes and takes $10.537$\,s on the
M4 Pro with 12 workers.

\subsection{The seventh coefficient}
The same method gives the independent proof of Theorem~\ref{thm:n7}.
The reduced $N=6$ search visits $1\,435\,479\,089$ nodes in 7.514 hours
($82.392$ CPU hours) on the M4 Pro with 12 workers.  The larger cost is
now the $N=5$ stratum, treated in atom coordinates: $67\,274\,389\,969$
nodes, with 33.684 hours ($390.160$ CPU hours) for the final M4 run.
That run resumes a laptop checkpoint; its counts include the earlier work,
but its timings do not.  The strata $N\le4$ close in under 8 minutes on
an M1 Pro.  Full counts, source identities and validation details are in
\cite[Section~S11]{Supplement}.

\begin{proof}[Proof of Theorem~\ref{thm:n7}]
An $M_7$-extremal has at most seven atoms.  The atom-coordinate searches
close covers of $R_N$ for $N=1,\dots,5$, so Proposition~\ref{prop:correct}
excludes these strata.  The reduced search closes the critical stratum
$N=6$ by the coverage argument of Section~\ref{sec:reduced}.  Thus every
extremal has seven atoms.  Theorem~\ref{thm:am-lz} now gives the coefficient
bound and the equality statement.
\end{proof}

\subsection{The numerical landscape}\label{sec:landscape}
To put the margins in perspective, Table~\ref{tab:landscape} lists the strict
local maxima of $\abs{a_6}$ over measures with at most six atoms found by
$10^5$ random starts of a local optimizer.  Each was checked, in floating
point, for a Hessian that is negative definite modulo rotation within its
stratum (a fixed number of atoms), and for the KKT condition $h\ge0$, which
controls the creation of new atoms to first order only.  These checks do not
certify local maximality among all measures, and the list need not be
complete.
The best configuration found with at most five atoms has $\abs{a_6}=0.700865$, so the
stratum $N=5$ has to be resolved with a margin of about $0.035$.  This
configuration has one atom of mass $0.774$ and four of mass about $0.18$,
arranged symmetrically.  It satisfies the relations of
Theorem~\ref{thm:kkt}, with $\abs{\sigma_1}=0.728$, so these relations alone
cannot exclude the stratum $N=5$; the search combines them with the
requirement $\abs{a_6}\ge2/e$ (tests P and V).  The analogous runners-up for $n=2,\dots,8$ all have
$n-1$ atoms and gaps of approximately $0.21/n$ to $2/e$.

\begin{table}[ht]
\centering
\begin{tabular}{c c c c}
\toprule
$\abs{a_6}$ & $2/e-\abs{a_6}$ & atoms & $t$\\
\midrule
0.735758882 & 0 & 6 & 1\\
0.700865261 & 0.0349 & 5 & 1.4804\\
0.663004776 & 0.0728 & 4 & 2.0000\\
0.661875816 & 0.0739 & 4 & 2.0166\\
0.657131510 & 0.0786 & 4 & 2.4374\\
0.618009572 & 0.1177 & 3 & 2.6180\\
0.612551107 & 0.1232 & 3 & 3.0688\\
0.605226130 & 0.1305 & 3 & 3.9500\\
0.551908447 & 0.1839 & 2 & 4.3580\\
0.550037386 & 0.1857 & 2 & 4.7831\\
0.541028711 & 0.1947 & 2 & 6.2158\\
0.449313548 & 0.2864 & 1 & 9.8373\\
\bottomrule
\end{tabular}
\medskip
\caption{Strict local maxima of $\abs{a_6}$ found numerically (floating point,
not certified).  Some are compositions of lower-order maxima with $z^k$: for
example $0.618009572=2\phi^3e^{-\phi^2}$, $\phi=(1+\sqrt5)/2$, comes from the
$n=2$ runner-up composed with $z^3$.}
\label{tab:landscape}
\end{table}

\subsection{Verification and reproducibility}\label{sec:verification}
The computational proof has two parts: valid exclusion bounds on each box,
and a search that covers the whole parameter domain.  For the first,
\cite{Enclosures} proves the floating-point bounds used by the atom-coordinate C code,
function by function, under its stated rounding assumptions.  Ball-arithmetic
scripts certify the derivative constants and mass cutoffs; additional scripts
check the numerical constants in the enclosure proofs and inspect the
compiled binary for the properties that the rounding model assumes.  Thus the validity of the tests rests on
explicit error bounds, rather than on agreement with sampled values.

For coverage of the atom-coordinate search, a separate script audits the logs.  It checks the root domain,
the bisections relating jobs, the leaf counts at each depth through a Kraft
identity, and completion without undecided boxes.  The leaves within each job
are not logged individually: their coverage rests on the depth-first search,
about 30 lines of C that push both halves of every bisected box.  The audit
therefore checks the consistency of the recorded search, not the numerical
bounds at its leaves.  The precise checks and their scope are given in
\cite[Section~S6]{Supplement}.

The original software, certification scripts and complete logs are available in
\cite{ZenodoV1}.  Empirical checks compare bounds with an independent Python
implementation, reproduce the known cases on a second
machine with a different compiler, test interrupted and resumed runs, and
check for memory errors and data races.  These checks support the
implementation; the proof uses the enclosure bounds and the coverage
argument above.  Full details are in
\cite[Section~S8]{Supplement}.

For the reduced search, the enclosure arguments and native execution
model are given in \cite[Sections~S9
and~S10]{Supplement}.  The production loop uses
directed binary64 arithmetic; Arb is used offline for certification and
independent checks.  SVD and LP calculations only propose data that the
enclosure kernel must verify.  LP helpers run in separate processes, so
their memory is not shared with the search or certificate verifier.

For the reduced $n=7$, $N=6$ search, the structural audit reads every node
record, checking the canonical root, subset containment for contractions,
exact bisections, batch frontiers, file hashes and the empty final frontier.
Independent Python/Arb checks at 256-bit precision resolve all $9\,331$
exported leaf samples and $3\,944$ contraction samples, in some cases by
subdividing and excluding the entire input box.  The smaller-stratum logs
also pass their structural audits.  These checks support the implementation;
they do not constitute independent arithmetic replay of every leaf.
The proof rests on the enclosure bounds and coverage arguments above.
The supplement records the scope of the audits and the retained artifacts.
The revised software, documentation and validation records are deposited
in \cite{Zenodo}.  This release includes complete reduced-search logs for
$n=5,6$ and compact validation evidence for $n=7$, $N=6$; the bulk node logs
for the latter run are retained separately by the author and are available
by arrangement.

\section{Related results and further directions}\label{sec:outlook}

The reduction in Section~\ref{sec:reduced} applies to the critical stratum
for general $n$.  The seventh-coefficient computation shows why reducing
only that stratum may not suffice for an efficient proof: the next stratum
can cost more.  Extending mass elimination to $N<n-1$ requires retaining
several auxiliary factors in~\eqref{eq:kkt-power}, together with their own
parameter bounds and coverage analysis.

Instanes, Llinares and Vukoti\'c \cite{ILV2026} relate the Krzy\.z conjecture
to weighted coefficient extremal problems.  Their implication from
$M_n=2/e$ to those problems applies at both $n=6$ and $n=7$ by
Corollary~\ref{cor:krzyz6} and Theorem~\ref{thm:n7}.



\begin{thebibliography}{99}

\bibitem{AglerMcCarthy2021}
J.~Agler and J.~E.~McCarthy, \emph{The Krzy\.z conjecture and an entropy
conjecture}, J.\ Anal.\ Math.\ \textbf{144} (2021), 207--226.

\bibitem{Brown1987}
J.~E.~Brown, \emph{Iteration of functions subordinate to schlicht functions},
Complex Variables Theory Appl.\ \textbf{9} (1987), 143--152.

\bibitem{Ermers1990}
R.~Ermers, \emph{Coefficient estimates for bounded nonvanishing functions},
PhD thesis, Katholieke Universiteit Nijmegen, 1990.

\bibitem{Horowitz1978}
C.~Horowitz, \emph{Coefficients of nonvanishing functions in $H^\infty$},
Israel J.\ Math.\ \textbf{30} (1978), 285--291.

\bibitem{HSZ1977}
J.~A.~Hummel, S.~Scheinberg and L.~Zalcman, \emph{A coefficient problem for
bounded nonvanishing functions}, J.\ Anal.\ Math.\ \textbf{31} (1977),
169--190.

\bibitem{ILV2026}
S.~M.~Instanes, A.~Llinares and D.~Vukoti\'c,
\emph{On extremal problems related to the Krzy\.z conjecture},
arXiv:2610.23817v2 (2026).

\bibitem{Arb}
F.~Johansson, \emph{Arb: efficient arbitrary-precision midpoint-radius
interval arithmetic}, IEEE Trans.\ Comput.\ \textbf{66} (2017), 1281--1292.

\bibitem{Krushkal2009}
S.~L.~Krushkal, \emph{Hyperbolic distances, nonvanishing holomorphic functions
and Krzy\.z's conjecture}, arXiv:0908.2587 (2009).

\bibitem{Krushkal2016}
S.~L.~Krushkal, \emph{Hyperbolic geodesics, Krzy\.z's conjecture and beyond},
arXiv:1603.02668 (2016).

\bibitem{Krushkal2020}
S.~L.~Krushkal, \emph{A new look at Krzy\.z's conjecture}, arXiv:2003.14214
(2020).

\bibitem{Krushkal2022}
S.~L.~Krushkal, \emph{Two coefficient conjectures for nonvanishing Hardy
functions, I}, J.\ Math.\ Sci.\ \textbf{268} (2022), 199--221.

\bibitem{Krzyz1968}
J.~Krzy\.z, \emph{Coefficient problem for bounded nonvanishing functions},
Ann.\ Polon.\ Math.\ \textbf{20} (1968), 314.

\bibitem{LeiZhang2026}
J.~Lei and T.~Zhang, \emph{Proof of the Agler--McCarthy entropy conjecture},
arXiv:2605.03949 (2026).

\bibitem{MacDonald2026}
S.~F.~MacDonald, \emph{Structural aspects of extremal functions in the Krzy\.z
conjecture}, arXiv:2605.08411 (2026).

\bibitem{MacDonaldConvolution2026}
S.~F.~MacDonald, \emph{A convolution inequality and the Krzy\.z conjecture},
arXiv:2611.05363 (2026).

\bibitem{MakinoBerz2003}
K.~Makino and M.~Berz, \emph{Taylor models and other validated functional
inclusion methods}, Int.\ J.\ Pure Appl.\ Math.\ \textbf{4} (2003), 379--456.

\bibitem{MSUV2015}
M.~J.~Mart\'in, E.~T.~Sawyer, I.~Uriarte-Tuero and D.~Vukoti\'c,
\emph{The Krzy\.z conjecture revisited}, Adv.\ Math.\ \textbf{273} (2015),
716--745.

\bibitem{Rockafellar1970}
R.~T.~Rockafellar, \emph{Convex Analysis}, Princeton University Press, 1970.

\bibitem{Rump2006}
S.~M.~Rump, \emph{Verification of positive definiteness}, BIT \textbf{46}
(2006), 433--452.

\bibitem{Samaris2003}
N.~Samaris, \emph{A proof of Krzy\.z's conjecture for the fifth coefficient},
Complex Var.\ Theory Appl.\ \textbf{48} (2003), 753--766.

\bibitem{Tan1983}
D.-L.~Tan, \emph{Coefficient estimates for bounded nonvanishing functions},
Chinese Ann.\ Math.\ Ser.~A \textbf{4} (1983), 97--104 (Chinese).

\bibitem{WikstromV1}
F.~Wikström, \emph{A computer-assisted proof of Krzy\.z's conjecture for
the sixth coefficient}, first version,
\href{https://arxiv.org/abs/2610.15206v1}{arXiv:2610.15206v1} (2026).

\bibitem{Xie2026}
X.~Xie, \emph{Krzy\.z's conjecture for the seventh coefficient},
\href{https://arxiv.org/abs/2611.00397v1}{arXiv:2611.00397v1} (2026).

\bibitem{Supplement}
F.~Wikström, \emph{Implementation and validation of the Krzy\.z searches},
supplementary material (2026), in \cite{Zenodo}; original version in
\cite{ZenodoV1}.

\bibitem{Enclosures}
F.~Wikström, \emph{Rigorous enclosures in the branch-and-bound for the sixth
Krzy\.z coefficient}, supplementary material, in \cite{Zenodo}.

\bibitem{ZenodoV1}
F.~Wikström, \emph{Software and logs for ``A computer-assisted proof of
Krzy\.z's conjecture for the sixth coefficient''}, Zenodo (2026), \href{https://doi.org/10.5281/zenodo.23288285}{doi:10.5281/zenodo.23288285}.

\bibitem{Zenodo}
F.~Wikström, \emph{Software and validation records for ``A computer-assisted
proof of Krzy\.z's conjecture for the sixth coefficient''}, version 2.0,
Zenodo (2026),
\href{https://doi.org/10.5281/zenodo.23964169}{doi:10.5281/zenodo.23964169}.

\end{thebibliography}
\end{document}